\documentclass[10pt]{amsart}
\usepackage{mathpazo}
\usepackage{amssymb}
\usepackage{amsmath}
\usepackage{graphicx}
\usepackage{hyperref}
\hypersetup{
colorlinks=true,
linkcolor=blue,
filecolor=blue,
urlcolor=blue,
citecolor=blue,
pdfpagemode=Fullscreen
}
\usepackage{xcolor}
\usepackage{soul}
\DeclareFontFamily{U}{mathx}{\hyphenchar\font45}
\DeclareFontShape{U}{mathx}{m}{n}{
      <5> <6> <7> <8> <9> <10>
      <10.95> <12> <14.4> <17.28> <20.74> <24.88>
      mathx10
      }{}
\DeclareSymbolFont{mathx}{U}{mathx}{m}{n}
\DeclareFontSubstitution{U}{mathx}{m}{n}
\DeclareMathAccent{\widecheck}{0}{mathx}{"71}

\newtheorem{theorem}{Theorem}[section]
\newtheorem{lemma}[theorem]{Lemma}
\newtheorem*{Acknowledgement}{\textnormal{\textbf{Acknowledgement}}}
\theoremstyle{definition}
\newtheorem{definition}[theorem]{Definition}

\newtheorem{corollary}[theorem]{Corollary}
\newtheorem{proposition}[theorem]{Proposition}
\newtheorem{proof of Theorem 1.1}[theorem]{Proof of Theorem 1.1}

\numberwithin{equation}{section}
\newcommand{\beqa}{\begin{eqnarray*}}
\newcommand{\eeqa}{\end{eqnarray*}}
\newcommand{\beqn}{\begin{eqnarray}}
\newcommand{\eeqn}{\end{eqnarray}}

\newcounter{cnt1}
\newcounter{cnt2}
\newcounter{cnt3}
\newcommand{\blr}{\begin{list}{$($\roman{cnt1}$)$}
        {\usecounter{cnt1} \setlength{\topsep}{0pt}
                \setlength{\itemsep}{0pt}}}
\newcommand{\bla}{\begin{list}{$($\alph{cnt2}$)$}
        {\usecounter{cnt2} \setlength{\topsep}{0pt}
                \setlength{\itemsep}{0pt}}}
\newcommand{\bln}{\begin{list}{$($\arabic{cnt3}$)$}
        {\usecounter{cnt3} \setlength{\topsep}{0pt}
                \setlength{\itemsep}{0pt}}}
\newcommand{\el}{\end{list}}
\newtheorem{thm}{Theorem}

\newtheorem{Def}[thm]{Definition}

\newtheorem{rem}[thm]{Remark}
\newcommand{\Rem}{\begin{rem} \rm}
\newcommand{\bdfn}{\begin{Def} \rm}
\newcommand{\edfn}{\end{Def}}

\begin{document}

\title[\tiny{Bijections on the set of extreme points in a compact convex set}]{Bijections on the set of extreme points in a compact convex set}


\author[Karn, Seal]{Anil Kumar Karn$^{1}$,\ Susmita Seal$^{2}$}

\address{{$^{1}$} School of Mathematical Sciences, National Institute of Science Educational and Research Bhubaneswar, An OCC of Homi Bhabha National Institute, P.O. - Jatni, District - Khurda, Odisha - 752050, India}
\email{anilkarn@niser.ac.in}

\address {{$^{2}$} School of Mathematical Sciences, National Institute of Science Educational and Research Bhubaneswar, An OCC of Homi Bhabha National Institute, P.O. - Jatni, District - Khurda, Odisha - 752050, India}
		\email{susmitaseal1996@gmail.com}



\maketitle
\begin{abstract}
In a recent work, Roelands and Tiersma \cite{RT} proved that, for a compact convex set $K$, the space $A(K)$ of all real-valued continuous affine functions on $K$, is a JB-algebra if and only if there is a gauge-reversing bijection on $A_c(K)$, the set of positive real-valued continuous affine functions on $K$. In this paper, we show that every such gauge-reversing bijection on $A_c(K)$ is completely determined by the induced bijection on the set of extreme points of $K$.
\end{abstract}

\section{Introduction and Preliminaries}

Gelfand-Naimark theorem states that any unital, abelian C$^*$-algebra is precisely $C(K)$, the space of all conplex valued continuous functions on $K$, where $K$ is a suitable compact space \cite{GN}. However, Sherman \cite{S51} prove that this characterization fails to work if the C$^*$-algebra is non-abelian. Kadison took this discussion further and proved that the self-adjoint part of a unital C$^*$-algebra, or more generally, the self-adjoint part of any self-adjoint unital subspace of it, is $A(K)$, the space of all affine real valued continuous functions on $K$, where $K$ is a suitable compact, convex subset $K$ of a locally convex space \cite{K1}. It is also known that every unital JB-algebra is $A(K)$ for some suitable compact, convex subset $K$ of a locally convex space \cite{A, ASS}. Owing to this observation, compact convex sets become a central object in the interface of non-abelian operator algebras and order theoretic functional analysis. Recently, Roelands and Tiersma have given a characterization of unital JB-algebras among $A(K)$-spaces using gauge-reversing maps. 

Let $K$ be a compact convex set, and consider $A(K)$. Due to a fundamental result of Kadison \cite{K1}, $A(K)$ precisely a complete order unit space up to order isomorphism. This representation allows one to study order-theoretic properties through the geometry of affine functions. Let 
$$A_c(K) = \lbrace f \in A(K): f(x) \in (0, \infty) ~ \textrm{for all} ~ x \in K \rbrace.$$ 
Note that for any $f, g \in A_c(K)$, there exists $\lambda > 0$ such that $f \leqslant \lambda g$. The gauge, which quantify the relative size of elements of $A_c(K)$, are defined as follows: 
 
$$m(f, g) = \sup\{\lambda > 0 : \lambda g \leq f\}, \quad
M(f, g) = \inf\{\mu > 0 : f \leq \mu g\} .$$
These functions are multiplicative inverses of each other, that is, $m(g,f)=M(f,g)^{-1}$ for all $f,g \in A_c(K)$. Moreover, the gauge function provides the necessary foundation for defining a natural distance function on $A_c(K)$, known as Thompson's metric. See \cite{NS} for details.
\begin{definition}\label{def1}
\cite{NS} Let $K$, K' be two compact convex sets.
A bijection $\Phi:A_c(K) \rightarrow A_c(K')$ is called 
\begin{enumerate}
\item gauge-preserving if for all $f, g \in A_c(K)$,  $$m(\Phi(f), \Phi(g)) = m(f, g) \ \ (\mathrm{equivalently} \ M(\Phi(f), \Phi(g)) = M(f, g)).$$ 
\item gauge-reversing if for all $f, g \in A_c(K)$, 
 $$m(\Phi(f), \Phi(g)) = m(g, f) \ \ (\mathrm{equivalently} \  M(\Phi(f), \Phi(g)) = M(g, f)).$$ 
\end{enumerate}
\end{definition}
Gauge-preserving bijections are of particular interest because of their linear structure. Any linear isomorphism between $A_c(K)$ and $A_c(K')$ is gauge-preserving, and, conversely, Noll and Schäffer \cite{NS} showed that every gauge-preserving bijection is necessarily a linear isomorphism. In contrast, gauge-reversing bijections on $A_c(K)$ play a significant role in the characterization of JB-algebras within $A(K)$ spaces (see \cite{RT}). 

The main objective of this paper is to show that both gauge-preserving and gauge-reversing bijections on $A_c(K)$ are completely determined by the induced bijections between the extreme points of the underlying compact convex sets.

The problem of characterizing algebraic structures within $A(K)$ spaces has a rich history and has been studied extensively in recent literature. Combining Kadison’s representation theorem \cite{A}, results of Walsh \cite{W}, and the Koecher-Vinberg theorem \cite{K,V}, one obtains that
 if $\partial_e(K)$, the set of extreme vectors of $K$, is finite, then a gauge-reversing bijection $\Phi: A_c(K)\rightarrow A_c(K)$ exists if and only if $A(K)$ is a JB-algebra.
Extensions of this result to the infinite-dimensional setting were developed in \cite{LRI, LRW}, where, under additional assumptions, such bijections were shown to exist precisely when
$A(K)$ belongs to a certain class of reflexive JB-algebras, such as spin factors and JH-algebras. A complete characterization was recently obtained in \cite[Theorem 1.1]{RT}, establishing that a gauge-reversing bijection $\Phi: A_c(K)\rightarrow A_c(K)$ exists if and only if $A(K)$ is a JB-algebra. 

Both gauge-preserving and gauge-reversing bijections admit a characterization in terms of homogeneity and order structure (see  \cite{LRW,NS}). Moreover, maps of both types are infinitely differentiable \cite{RT}. A prime example of a gauge-reversing bijection arises from the inverse map, whenever the inverse exists--for instance, in the setting of a JB-algebra. Suppose $A(K)$ is a JB-algebra. Then $A_c(K)=\{f^2: f\in A(K), f \ne 0 \}$ and $\frac{1}{g}\in A_c(K)$ for each $g\in A_c(K)$. Consequently, the map $\Phi: A_c(K)\rightarrow A_c(K)$ defined by $g\mapsto \frac{1}{g}$ is a gauge-reversing bijection. In this paper, we prove the converse.

\begin{theorem}\label{coro}
	Let $K$ and $K'$ be two compact convex sets.
	The following statements are equivalent: 
	\begin{enumerate}
		\item There exists a gauge-reversing bijection $\Phi : A_c(K) \to A_c(K')$ such that $\Phi(\mathbb{1}_K) = \mathbb{1}_{K'}.$
		\item There exists a bijection $\alpha : \partial_e K \to \partial_{e} K'$ such that for every $f \in A_c(K)$ and $g \in A_c(K')$, the functions $\frac{1}{g} \circ \alpha$ and $\frac{1}{f} \circ \alpha^{-1}$ are continuous and admit continuous affine extensions to the respective extreme boundaries.
	\end{enumerate}
	Moreover, when these equivalent conditions hold, $\Phi$ and its inverse are completely determined on the extreme boundaries by
	\begin{equation}\label{gr}
		\begin{split}
			\Phi(f)|_{\partial_{e} K'} = \frac{1}{f} \circ \alpha^{-1} \quad \text{for every } f \in A_c(K)\\
			\Phi^{-1}(g)|_{\partial_e K} = \frac{1}{g} \circ \alpha \quad \text{for every } g \in A_c(K').
		\end{split}
	\end{equation}
\end{theorem}
We also prove 
\begin{theorem}\label{precoro}
	Let $K$ and $K'$ be two compact convex sets.
	The following statements are equivalent: 
	\begin{enumerate}
		\item There exists a gauge-preserving bijection $\Phi : A_c(K) \to A_c(K')$ such that $\Phi(\mathbb{1}_K) = \mathbb{1}_{K'}.$ 
		\item There exists a bijection $\alpha : \partial_e K \to \partial_{e} K'$ such that for every $f \in A_c(K)$ and $g \in A_c(K')$, the functions $g \circ \alpha$ and $f \circ \alpha^{-1}$ are continuous and admit continuous affine extensions to the respective extreme boundaries. 
	\end{enumerate}
	Moreover, when these equivalent conditions hold, $\Phi$ and its inverse are completely determined on the extreme boundaries by
	\begin{equation}\label{gp}
		\begin{split}
			\Phi(f)|_{\partial_{e} K'} = f \circ \alpha^{-1} \quad \text{for every } f \in A_c(K) \\
			\Phi^{-1}(g)|_{\partial_e K} = g \circ \alpha \quad \text{for every } g \in A_c(K')
		\end{split}
	\end{equation}
\end{theorem}

In the course of our work, we make use of the notion of extended cones of $A_c(K)$, introduced in \cite{RT}. We denote by $A(K)_+^\infty$ the collection of all affine functions from 
$K$ into the extended interval $[0, \infty]$.
This set forms a monotone complete partially ordered set with respect to 
the natural pointwise order: for 
$f,g\in A(K)_+^\infty$ we write $f \leqslant g$ if $f(x) \leq g(x)$ for every $x \in K$. 
The set $A(K)_+^\infty$ is also equipped with pointwise addition and scalar multiplication by elements of $[0, \infty]$, where we adopt the convention $0 \cdot \infty = 0$. The two extended cones of $A_c(K)$ in $A(K)_+^\infty$ are following: 
$$A_{usc}(K) := \{g : K \to [0, \infty) : g \  \text{is affine upper semicontinuous}\}$$
$$A_{lsc}(K) := \{h : K \to (0, \infty] : g \ \text{is affine lower semicontinuous}\}.$$
The cone $A_{lsc}(K)$ is closed under addition and multiplication by scalars in $(0, \infty],$ while $A_{usc}(K)$ is closed under addition and multiplication by scalars in $[0, \infty)$. Moreover, we have that
$$A_{usc}(K) \cap A_{lsc}(K) = A_c(K).$$
The gauge functions $m$ and $M$ from Definition $\ref{def1}$ admit natural extensions to $A(K)_+^\infty$, see \cite[Definition 3.9]{RT} for details.
Moreover, each gauge-reversing (resp. gauge-preserving) bijection between $A_c(K)$ and $A_c(K')$ admit unique extensions to the
corresponding extended cones \cite[Theorem 3.14, 3.15]{RT}.

We shall now give an outline of the organization of the paper.
In section $\ref{Sec3}$ we show that any extremal vector of $A_{usc}(K)$  is of the form $p_{\psi}:= \inf \{h\in A_{usc}(K) : h(\psi)=1\}$ for some $\psi \in \partial_e K.$ Using this description, we prove that $A_{usc}(K)$ is strongly atomic, generalizing \cite[Theorem 3.20]{RT}, where the authors established this property under the assumption that a gauge-reversing bijection exists between the cones.
In section $\ref{Sec4}$ we introduce the notion of a co-extremal vector and show that any co-extremal vectors of $A_{lsc}(K)$ is of the form $I_{\{\psi\}}^{\infty}$ for some $\psi \in \partial_e(K)$, where $I_{\{\psi\}}^{\infty}(\rho) =1$ if $\rho = \psi$ and $I_{\{\psi\}}^{\infty}(\rho) =0$ if $\rho \neq \psi.$ 
Finally, in section $\ref{Sec5}$ we prove the main result of the paper, that is Theorems \ref{coro} and \ref{precoro}, and in section $\ref{Sec6}$ we discuss some applications.

\section{Extremal vectors in $A_{usc}(K)$} \label{Sec3}

Let $E \subset A(K)_{+}^{\infty}$ be a subcone.
A vector $p \in E$ is called extremal vector of $E$ if  $$\{f \in E : f \leqslant p\} = \{tp : 0 \leqslant t \leqslant 1\}.$$ The collection of all extremal vectors in $E$ is denoted $\text{ext}\,E$. The cone $E$ is said to be strongly atomic if every element $f \in E$ can be expressed as
 $f = \sup\{p \in \text{ext}\,E : p \leq f\} \text{ in } E.$ Furthermore, an extremal vector $p$ in $A_{usc}(K)$ satisfying $M(p, \mathbb{1}_K) = 1$ is called a normalized extreme vector. The set of all such vectors is denoted by $$\text{ext}_{\mathbb{1}_K} A_{usc}(K) := \{p \in \text{ext}\,A_{usc}(K) : M(p, \mathbb{1}_K) = 1\}.$$

Let $D$ be a nonempty subset of $A(K)_{+}^{\infty}$. Then $D$ is called a nonincreasing net (resp. nondecreasing net) indexed by $(D,\leqslant)$ if
for any
$g_1,g_2 \in D$ there exists $g \in D$ such that $g \leqslant g_i$
(resp. $g \geqslant g_i$) for $i=1,2$. 

\begin{theorem}\label{noninc}
\cite[Lemma 3.4]{RT}
Let $K$ be a compact convex set.
\begin{enumerate}
    \item Each nonincreasing net in $A_{usc}(K)$ has an infimum in $A_{usc}(K)$, which is given pointwise.
    \item Each nondecreasing net in $A_{lsc}(K)$ has a supremum in $A_{lsc}(K)$, which is given pointwise.
\end{enumerate}
\end{theorem}


\begin{theorem}\label{char ext}
Every element of $\mathrm{ext}_{\mathbb{1}_K} A_{usc}(K)$ is of the form
\[
p_\psi := \inf \{h \in A_{usc}(K) : h(\psi)=1\}
\]
for some $\psi \in \partial_e K$.
\end{theorem}

\begin{proof}
Let $\psi \in \partial_e K$ and consider the set
\[
\mathcal{A} := \{h \in A_{usc}(K) : h(\psi)=1\}.
\]
We first show that $p_\psi \in A_{usc}(K)$. For this it suffices to prove that
$\mathcal{A}$ forms a nonincreasing net. Take $h_1,h_2 \in \mathcal{A}$ and define
\[
M_1=\{(x,t)\in K\times \mathbb{R}: t>h_1(x)\}, \qquad
M_2=\{(x,t)\in K\times \mathbb{R}: t>h_2(x)\},
\]
and
\[
M=\{(\psi,\varepsilon): 0\le \varepsilon\le 1\}.
\]

Note that $(M_1\cup M_2)\cap M=\varnothing$, and hence also
\[
\mathrm{co}(M_1\cup M_2)\cap M=\varnothing .
\]
Indeed, suppose $(\psi,\varepsilon)\in \mathrm{co}(M_1\cup M_2)$. Then
\[
(\psi,\varepsilon)=\sum_{i=1}^n \lambda_i (x_i,t_i),
\]
where $(x_i,t_i)\in M_1\cup M_2$. Since $\psi$ is an extreme point of $K$,
we must have $x_i=\psi$ for all $i$. Consequently,
\[
\varepsilon=\sum_{i=1}^n \lambda_i t_i>1,
\]
because $(\psi,t_i)\in M_1\cup M_2$ implies $t_i>1$. This contradicts
$0\leqslant \varepsilon\le 1$, and therefore the convex hull is disjoint from $M$.

By the separation theorem there exists a hyperplane
\[
H(x,t)=\varphi(x)+\lambda t,
\]
with $\varphi\in V^*$ and $\lambda\in\mathbb{R}$, such that
\[
\inf_{(x,t)\in \mathrm{co}(M_1\cup M_2)^\circ} H(x,t)
\ge
\sup_{(x,t)\in M} H(x,t).
\]
Hence there exists $c\in\mathbb{R}$ satisfying
\[
H(x,t)>c \qquad \forall (x,t)\in \mathrm{co}(M_1\cup M_2)^\circ,
\tag{3.1}
\]
and
\[
H(\psi,\varepsilon)\le c \qquad \forall\, 0\le \varepsilon\le 1.
\tag{3.2}
\]

We claim that $\lambda\neq 0$. If $\lambda=0$, then (3.1) yields
$\varphi(x)>c$ for all $x\in K$, while (3.2) implies
$\varphi(\psi)\le c$, which is impossible.

Define
\[
h(x)=\frac{c-\varphi(x)}{\lambda}, \qquad x\in K.
\]
From (3.1), for any $\delta>0$ we have
\[
H(x,h_1(x)+\delta)>c \quad \forall x\in K,
\]
which implies
\[
\varphi(x)+\lambda(h_1(x)+\delta)>c.
\]
Thus $h_1(x)+\delta>h(x)$ for all $x$, and therefore $h_1\geqslant h$.
An identical argument shows $h_2\geqslant h$. Moreover,
$h(\psi)\leqslant h_1(\psi)=1$. On the other hand, using (3.2) with
$\varepsilon=1$ gives $h(\psi)\geqslant 1$. Hence $h(\psi)=1$, so $h\in \mathcal{A}$.

This proves that $\mathcal{A}$ is nonincreasing net, and therefore by Theorem $\ref{noninc}$,
$p_\psi=\inf \mathcal{A}$ belongs to $A_{usc}(K)$.

Now let $p\in \mathrm{ext}_{\mathbb{1}_K} A_{usc}(K)$. Choose
$\psi\in \partial_e K$ such that
\[
\sup_{x\in K} p(x)=p(\psi).
\]
Since $p$ is normalized, $\sup\limits_{x\in K} p(x)=M(p,{\mathbb{1}_K})=1$, and hence
$p(\psi)=1$. Because $p\geqslant p_\psi$, extremality implies
$p_\psi=\lambda p$ for some $0<\lambda\leqslant 1$. Evaluating at $\psi$ gives
\[
1=p_\psi(\psi)=\lambda p(\psi)=\lambda,
\]
and therefore $p=p_\psi$.
\end{proof}

\begin{proposition}\label{ext prop}
For a compact convex set $K$, we have the following.
\begin{enumerate}
\item If $\psi \neq \tilde{\psi}$ are elements of $\partial_e K$, then
$p_\psi \neq p_{\tilde{\psi}}$.

\item The function $p_\psi$ satisfies
\[
p_\psi(\psi)=1 \quad \text{and} \quad
p_\psi(\rho)=0 \quad \text{for all } \rho \in \partial_e K\setminus\{\psi\}.
\]
\end{enumerate}
\end{proposition}

\begin{proof}
(i) Define
\[
f := \inf \{h \in A_{usc}(K) : h(\psi)=1,\; h(\tilde{\psi})<1\}.
\]
Arguing as in Proposition $\ref{char ext}$, the set
\[
\mathcal{A}=\{h \in A_{usc}(K) : h(\psi)=1,\; h(\tilde{\psi})<1\}
\]
forms a nonincreasing net, and hence $f \in A_{usc}(K)$.
Moreover $f(\psi)=1$ and $f(\tilde{\psi})<1$. Since $p_\psi \le f$, we obtain
\[
p_\psi(\tilde{\psi}) \leqslant f(\tilde{\psi}) < 1,
\]
which shows that $p_\psi \neq p_{\tilde{\psi}}$.

(ii) Let $\tilde{\psi} \in \partial_e K \setminus \{\psi\}$ and take
$0<r<1$. Define
\[
f_r := \inf \{h \in A_{usc}(K) : h(\psi)=1,\; h(\tilde{\psi})<r\}.
\]
Then $f_r \in A_{usc}(K)$ with $f_r(\psi)=1$ and $f_r(\tilde{\psi})<r$.
Since $p_\psi \leqslant f_r$ for every $0<r<1$, it follows that
\[
0 \le p_\psi(\tilde{\psi}) \le f_r(\tilde{\psi}) < r
\quad \text{for all } 0<r<1.
\]
Hence $p_\psi(\tilde{\psi})=0$.
\end{proof}

\begin{corollary}
For every $\rho \in \mathrm{co}(\partial_e K)$ one has
$\rho = \sum\limits_{i=1}^{n} p_{\psi_i}(\rho)\,\psi_i.$
\end{corollary}

\begin{proof}
Proof follows from Proposition $\ref{ext prop}$ (ii).
\end{proof}





\begin{proposition}\label{cusc boudary}
Let $h,h' \in A_{usc}(K)$. If
$h(\psi) \leqslant h'(\psi)$ for all $\psi \in \partial_e K,$
then $h \leqslant h'$.
\end{proposition}

\begin{proof}
Since $h' \in A_{usc}(K)$, there exists a decreasing net
$(f_\alpha)_{\alpha \in \Lambda}$ of continuous affine functions on $K$
such that
\[
h'(\rho) = \inf_{\alpha \in \Lambda} f_\alpha(\rho), \qquad \forall \rho \in K.
\]

Fix $\alpha \in \Lambda$ and define the upper semicontinuous affine
function
\[
g_\alpha := h - f_\alpha .
\]
Then for every $\psi \in \partial_e K$ we have $g_\alpha(\psi) \leqslant 0$.
Since an upper semicontinuous affine function attains its supremum
on $\partial_e K$ \cite[Lemma 3.7]{RT}, there exists $\psi_0 \in \partial_e K$ such that
\[
\sup_{\rho \in K} g_\alpha(\rho) = g_\alpha(\psi_0).
\]
Hence $\sup\limits_{\rho \in K} g_\alpha(\rho) \leqslant 0$, which implies
$g_\alpha \leqslant 0$. Therefore $h \leqslant f_\alpha$ for every $\alpha \in \Lambda$.
Taking the infimum over $\alpha$ gives $h \leqslant h'$.
\end{proof}

\begin{lemma}\label{lemrt}
\cite[Lemma 3.17]{RT}
The cone $A_{usc}(K)$ is strongly atomic if and only if for all
$g,g' \in A_{usc}(K)$ one has
\[
m(g,p) \leqslant m(g',p) \quad \forall p \in \mathrm{ext}_{\mathbb{1}_K} A_{usc}(K)
\quad \Longleftrightarrow \quad
g \leqslant g'.
\]
\end{lemma}

In \cite[Proposition 3.20]{RT}, it was proved that the cone $A_{usc}(K)$ is strongly atomic under the assumption that there exists a gauge-reversing bijection $\Phi:A_c (K) \rightarrow A_c (K')$. The following result removes this additional requirement and shows more generally that, for any compact convex set $K$, the cone $A_{usc}(K)$ is always strongly atomic.

\begin{theorem}\label{stratom}
$A_{usc}(K)$ is strongly atomic.
\end{theorem}

\begin{proof}
We first show that for any $g \in A_{usc}(K)$ and
$\psi \in \partial_e K$, $m(g,p_\psi) = g(\psi).$
Indeed, observe that $m(g,p_\psi)p_\psi \leqslant g$, evaluating at $\psi$ gives
$m(g,p_\psi) \leqslant g(\psi)$.
On the other hand, by Proposition $\ref{ext prop}$ we have
\[
g(\psi)p_\psi|_{\partial_e K} \leqslant g|_{\partial_e K}.
\]
Applying Proposition $\ref{cusc boudary}$ yields $g(\psi)p_\psi \leqslant g$, and therefore
$m(g,p_\psi) \geqslant g(\psi)$. Combining both inequalities gives
$m(g,p_\psi) = g(\psi)$.

Finally, Proposition $\ref{char ext}$ together with Lemma $\ref{lemrt}$ implies that
$A_{usc}(K)$ is strongly atomic.
\end{proof}

\section{Co-extremal vectors in $A_{lsc}(K)$} \label{Sec4}

We begin by introducing the following notion.

\begin{definition}\label{defco}
Let $E \subset A(K)_{+}^{\infty}$ be a subcone. 
A vector $q \in E$ is called 
\emph{co-extremal vector} of $E$ if
\[
\{f \in E : f \geqslant q\} = \{tq : 1 \leqslant t \leqslant \infty\}.
\]
The collection of all co-extremal vectors of $E$ is denoted by
$\mathrm{co\text{-}ext}\,E$.

The cone $E$ is said to be \emph{strongly co-atomic} if every element
$f \in E$ satisfies
\[
f = \inf \{q \in \mathrm{co\text{-}ext}\,E : q \geqslant f\}
\quad \text{in } E.
\]

We further write
\[
\mathrm{co\text{-}ext}_{\mathbb{1}_K} A_{lsc}(K)
=
\{q \in \mathrm{co\text{-}ext}\,A_{lsc}(K) : M(\mathbb{1}_K,q)=1\}.
\]
\end{definition}

For $\psi \in \partial_e(K)$, define the function
$I_{\{\psi\}}^{\infty} : K \to (0,\infty]$ by
\[
I_{\{\psi\}}^{\infty}(\rho)=
\begin{cases}
1, & \text{if } \rho=\psi,\\
\infty, & \text{if } \rho \ne \psi .
\end{cases}
\]
Observe that $I_{\{\psi\}}^{\infty}$ is an element in $A_{lsc}(K).$
\begin{theorem}\label{co-ext char}
Let $q \in \mathrm{co\text{-}ext}_{\mathbb{1}_K} A_{lsc}(K)$. Then there exists
$\psi \in \partial_e K$ such that $q = I_{\{\psi\}}^{\infty}.$
\end{theorem}

\begin{proof}
Take $q \in \mathrm{co\text{-}ext}_e A_{lsc}(K)$. Choose
$\psi \in \partial_e K$ for which $\inf\limits_{x \in K} q(x) = q(\psi).$
Since $q$ is normalized, we have
\[
\inf_{x \in K} q(x) = m(q,{\mathbb{1}_K}) = \frac{1}{M({\mathbb{1}_K},q)} = 1,
\]
and therefore $q(\psi)=1$.
Observe that $q \leqslant I_{\{\psi\}}^{\infty}$. By the co-extremality of $q$,
this implies that
$I_{\{\psi\}}^{\infty} = \lambda q$
for some scalar $\lambda$ with $1 \leqslant \lambda < \infty$. Evaluating at
$\psi$ gives
\[
1 = I_{\{\psi\}}^{\infty}(\psi) = \lambda q(\psi) = \lambda,
\]
and hence $\lambda=1$. Consequently $q = I_{\{\psi\}}^{\infty}$.
\end{proof}
We now establish that the notion of strongly co-atomicity can be characterized in a manner analogous to \cite[Lemma 3.7]{RT}.
\begin{proposition}\label{coatomic prop}
Let $K$ be a compact convex set. The following statements are equivalent:
\begin{enumerate}
\item $A_{lsc}(K)$ is strongly co-atomic.
\item For every $g \in A_{lsc}(K)$, we have
$g = \inf \{ M(g, I_{\{\psi\}}^{\infty})\, I_{\{\psi\}}^{\infty} : \psi \in \partial_e K \}.
$
\item For all $g,g' \in A_{lsc}(K)$,
\[
M(g, I_{\{\psi\}}^{\infty}) \leqslant M(g', I_{\{\psi\}}^{\infty})
\quad \text{for every } \psi \in \partial_e K
\quad \Longleftrightarrow \quad
g \leqslant g'.
\]
\end{enumerate}
\end{proposition}

\begin{proof}
(i) $\Leftrightarrow$ (ii).
By definition, $A_{lsc}(K)$ is strongly co-atomic if and only if
for every $g \in A_{lsc}(K)$,
\[
g = \inf \{f \in \mathrm{co\text{-}ext}\,A_{lsc}(K) : f \geqslant g\}.
\]
Using Proposition $\ref{coatomic prop}$, each co-extremal element has the form
$I_{\{\psi\}}^{\infty}$. Hence
\[
\{f \in \mathrm{co\text{-}ext}\,A_{lsc}(K) : f \geqslant g\}
=
\{\lambda q : q \in \mathrm{co\text{-}ext}_{\mathbb{1}_K} A_{lsc}(K),\ \lambda q \geqslant g\}.
\]
This is equivalent to
\[
\{\lambda q : \lambda \geqslant M(g,q),\ q \in \mathrm{co\text{-}ext}_{\mathbb{1}_K} A_{lsc}(K)\}.
\]
Therefore
\[
g = \inf \{M(g,q)\,q : q \in \mathrm{co\text{-}ext}_{\mathbb{1}_K} A_{lsc}(K)\},
\]
which yields
\[
g = \inf \{M(g, I_{\{\psi\}}^{\infty})\, I_{\{\psi\}}^{\infty} : \psi \in \partial_e K\}.
\]

(ii) $\Rightarrow$ (iii).
Suppose $g,g' \in A_{lsc}(K)$ satisfy
$M(g, I_{\{\infty_\psi\}}) \le M(g', I_{\{\psi\}}^{\infty})$ for all
$\psi \in \partial_e K.
$
Then representation (ii) implies $g \leqslant g'$.
Conversely, if $g \leqslant g'$, then
\[
g \leqslant g' \leqslant M(g', I_{\{\psi\}}^{\infty})\, I_{\{\psi\}}^{\infty}
\]
for all $\psi \in \partial_e K$, which gives
\[
M(g, I_{\{\psi\}}^{\infty}) \leqslant M(g', I_{\{\psi\}}^{\infty}).
\]

(iii) $\Rightarrow$ (ii).
Note that
\[
g \leqslant M(g, I_{\{\psi\}}^{\infty})\, I_{\{\psi\}}^{\infty}
\quad \text{for all } \psi \in \partial_e K.
\]
Let $h$ be any element satisfying
\[
h \leqslant M(g, I_{\{\psi\}}^{\infty})\, I_{\{\psi\}}^{\infty}
\quad \forall \psi \in \partial_e K.
\]
Then
\[
M(h, I_{\{\psi\}}^{\infty}) \leqslant M(g, I_{\{\psi\}}^{\infty})
\quad \forall \psi \in \partial_e K.
\]
By (iii) this implies $h \leqslant g$. Hence
$
g = \inf \{ M(g, I_{\{\psi\}}^{\infty})\, I_{\{\psi\}}^{\infty} : \psi \in \partial_e K \}.
$
\end{proof}

\begin{proposition} \label{extlsc}
\cite[Lemma 3.8]{RT}
Let $h,h' : K \to (-\infty,\infty]$ be lower semicontinuous affine
functions. If
$
h(\psi) \leqslant h'(\psi)$ for every $\psi \in \partial_e K,
$
then $h \leqslant h'$.
\end{proposition}

\begin{corollary}
$A_{lsc} (K)$ is strongly co-atomic.
\end{corollary}

\begin{proof}
Note that for every $\psi \in \partial_e K$ and $g \in A_{lsc} (K)$, we have
$
M(g, I_{\{\psi\}}^{\infty}) = g(\psi).
$
Let $g,g' \in A_{lsc} (K)$. Then
\[
M(g, I_{\{\psi\}}^{\infty}) \leqslant M(g', I_{\{\psi\}}^{\infty})
\quad \text{for all } \psi \in \partial_e K
\]
\[
\iff g(\psi) \leqslant g'(\psi) \qquad \forall \psi \in \partial_e K.
\]
By Proposition $\ref{extlsc}$ this is equivalent to $g \leqslant g'$. Hence,
Proposition $\ref{coatomic prop}$ implies that $A_{lsc} (K)$ is strongly co-atomic.
\end{proof}

We recall the following result.
\begin{proposition}\label{phiusc}
\cite[Proposition 3.18]{RT}
Let $K$ and $K'$ be two compact convex sets.
Suppose $\Phi : A_c(K) \to A_c(K')$ is a gauge-reversing bijection with $\Phi(\mathbb{1}_K)=\mathbb{1}_{K'}$. 
Then for every $p\in \text{ext}_{\mathbb{1}_K} A_{usc}(K)$ there exist $\psi\in \partial_{e} K'$ such that $\Phi_{usc}(p)=I^{\infty}_{\{\psi\}}$ and for all $g\in A_{usc}(K)$, $M(p,g)=\Phi_{usc}(g)(\psi).$
\end{proposition}

Using Proposition $\ref{char ext}$, we obtain the following refinement of Proposition $\ref{phiusc}.$

\begin{proposition} \label{phiuscmodi}
Let $K$ and $K'$ be two compact convex sets.
Suppose $\Phi : A_c(K) \to A_c(K')$ is a gauge-reversing bijection with $\Phi(\mathbb{1}_K)=\mathbb{1}_{K'}$. 
For $\tilde{\psi} \in \partial_e K$
the following statements hold:
\begin{enumerate}
\item $\Phi_{usc}(p_{\tilde{\psi}})=  (I^{\infty}_{\{\psi\}})$
for some $\psi \in \partial_{e} K'$.

\item For every $g \in A_{usc}(K)$,
\[
\Phi_{usc}(g)(\psi)
=
M(p_{\tilde{\psi}}, g)
= \frac{1}{m(g, p_{\tilde{\psi}})}
=\frac{1}{g(\tilde{\psi})}.
\]
\end{enumerate}
\end{proposition}

From Proposition $\ref{co-ext char}$ and Proposition $\ref{phiusc}$ it follows that $\Phi_{usc}(\text{ext}_{\mathbb{1}_K} A_{usc}(K))\subset \text{co-ext}_{\mathbb{1}_{K'}} A_{lsc}(K').$ Taking into account that $\Phi_{usc}^{-1}=(\Phi^{-1})_{lsc}$ and $\Phi_{lsc}^{-1}=(\Phi^{-1})_{usc}$, our next result shows that 
$\Phi_{usc}(\text{ext}_{\mathbb{1}_K} A_{usc}(K))= \text{co-ext}_{\mathbb{1}_{K'}} A_{lsc}(K')$ and $\Phi_{lsc}(\text{co-ext}_{\mathbb{1}_K} A_{lsc}(K))= \text{ext}_{\mathbb{1}_{K'}} A_{usc}(K').$

\begin{proposition}\label{philsc}
Let $K$ and $K'$ be two compact convex sets.
Suppose $\Phi : A_c(K) \to A_c(K')$ is a gauge-reversing bijection with $\Phi(\mathbb{1}_K)=\mathbb{1}_{K'}$. 
 For $\tilde{\psi} \in \partial_e K$
the following statements hold:
\begin{enumerate}
\item $\Phi_{lsc}(I^{\infty}_{\{\tilde{\psi}\}}) = p_\psi$
for some $\psi \in \partial_{e} K'$.

\item For every $g \in A_{lsc}(K)$,
\[
\Phi_{lsc}(g)(\psi)
=
m(I^{\infty}_{\{\tilde{\psi}\}}, g)
=
\frac{1}{g(\tilde{\psi})}.
\]
\end{enumerate}
\end{proposition}

\begin{proof}
(i) Observe that
\[
\{f \in A_{lsc}(K) : f \geqslant I^{\infty}_{\{\tilde{\psi}\}}\}
=
\{t I^{\infty}_{\{\tilde{\psi}\}} : 1 \leqslant t \leqslant \infty\}.
\]
Applying $\Phi_{lsc}$ yields
\[
\{h \in A_{usc}(K') : h \le \Phi_{lsc}(I^{\infty}_{\{\tilde{\psi}\}})\}
=
\Phi_{lsc}\big(\{t I^{\infty}_{\{\tilde{\psi}\}} : 1 \leqslant t \leqslant \infty\}\big)
=
\{t\,\Phi_{lsc}(I^{\infty}_{\{\tilde{\psi}\}}) : 0 \leqslant t \leqslant 1\}.
\]
Furthermore,
\[
M\big(\Phi_{lsc}(I^{\infty}_{\{\tilde{\psi}\}}),
\Phi(e)\big)
=
M(e, I^{\infty}_{\{\tilde{\psi}\}})=1.
\]
Hence $\Phi_{lsc}(I^{\infty}_{\{\tilde{\psi}\}})$ is a normalized
extreme vector of $A_{usc}(K')$, and therefore equals $p_\psi$
for some $\psi \in \partial_{e} K'$.

(ii) For $g \in A_{lsc}(K)$ we compute
\[
\frac{1}{g(\tilde{\psi})}
=
m(I^{\infty}_{\{\tilde{\psi}\}}, g)
=
m\big(\Phi_{lsc}(g),
\Phi_{lsc}(I^{\infty}_{\{\tilde{\psi}\}})\big)
=
m(\Phi_{lsc}(g), p_\psi)
=
\Phi_{lsc}(g)(\psi).
\]
\end{proof}

Using arguments similar to those in Propositions $\ref{phiusc}$, $\ref{phiuscmodi}$ and $\ref{philsc}$, we obtain the following result.
\begin{proposition}\label{prephilsc}
Let $K$ and $K'$ be two compact convex sets.
Suppose $\Phi : A_c(K) \to A_c(K')$ is a gauge-reversing bijection with $\Phi(\mathbb{1}_K)=\mathbb{1}_{K'}$.
The following statements hold:
\begin{enumerate}
\item For every $\tilde{\psi} \in \partial_e K$, there exists $\psi \in \partial_{e} K'$ such that
 $\Phi_{lsc}(I^{\infty}_{\{\tilde{\psi}\}}) = I^{\infty}_{\{\psi\}}.$
 Moreover, for any
$g \in A_{lsc}(K)$, we have
$
\Phi_{lsc}(g)(\psi)
=
g(\tilde{\psi}).
$
\item For every $\tilde{\psi} \in \partial_e K$, there exists $\psi \in \partial_{e} K'$ such that $\Phi_{usc}(p_{\tilde{\psi}}) = p_{\psi}.$
Moreover, for any $g \in A_{usc}(K)$, we have
$
\Phi_{usc}(g)(\psi)
=
g(\tilde{\psi}).
$
\end{enumerate}
\end{proposition}

\section{Proof of Theorem \ref{coro} and Theorem \ref{precoro}} \label{Sec5}

\begin{proposition}
\cite[Lemma 2.5]{LRW} Let $K$ and $K'$ be two compact convex sets.
A bijection $\Phi:A_c(K) \rightarrow A_c(K')$ is  gauge-preserving (resp. gauge-reserving) if and only if
\begin{enumerate}
\item $\Phi(\lambda f)=\lambda \Phi (f)$ (resp. $\Phi(\lambda f)=\frac{1}{\lambda} \Phi (f)$) for all $\lambda >0$
\item 
whenever $f,g \in A_c(K)$ satisfy $f\leqslant g$, one has $\Phi(f)\leqslant \Phi(g)$ (resp. $\Phi(g)\leqslant \Phi(f)$).
\end{enumerate}
\end{proposition}

We now have all the ingredients to prove our main theorems.


\begin{proof}[Proof of Theorem \ref{coro}]
(i) $\Rightarrow$ (ii).
Define $
\alpha : \partial_e K \to \partial_{e} K'
$ by $\tilde{\psi}\mapsto \psi,$ where $\psi$ is determined by $\Phi_{lsc}(I^{\infty}_{\{\tilde{\psi}\}}) = p_\psi$. Since $\Phi_{lsc}$ is injective, the map  $\alpha$ is also injective. Moreover, since $\Phi_{lsc}^{-1}=(\Phi^{-1})_{usc},$ for any $\psi \in \partial_{e} K'$ Proposition $\ref{phiuscmodi}$ ensures the existence of $\tilde{\psi}\in \partial_e K$ such that $(\Phi^{-1})_{usc}(p_{\psi})=  (I^{\infty}_{\{\tilde{\psi}\}}).$ 
Hence $\alpha$ is onto.

Now let $g\in A_c(K)$, and $\psi \in \partial_{e} K'$. Choose $\tilde{\psi} \in \partial_e K$ such that $\alpha(\tilde{\psi})=\psi.$ Then $\Phi_{lsc}(I^{\infty}_{\{\tilde{\psi}\}}) = p_\psi$. By Proposition $\ref{philsc}$ we obtain
\[
\Phi(g)(\psi)
=\Phi_{lsc}(g)(\psi)
=
\frac{1}{g(\tilde{\psi})}
= \frac{1}{g(\alpha^{-1}(\psi))}.
\]
Hence, $\Phi(g)|_{\partial_{e} K'} = \frac{1}{g} \circ \alpha^{-1}.$ A similar argument using Proposition $\ref{phiuscmodi}$ yields  $$
\Phi^{-1}(h)|_{\partial_{e} K} = \frac{1}{h} \circ \alpha \quad \text{for all}\ h\in A_c(K').$$



(ii) $\Rightarrow$ (i).
Observe that by Proposition $\ref{extlsc}$ (or $\ref{cusc boudary}$), whenever such an extension exists then it is unique.
 Define $\Phi:A_c(K)\rightarrow A_c(K')$ by $$\Phi(f)=\frac{1}{f} \circ \alpha^{-1}, \qquad f\in A_c(K).$$
 Let $f_1,f_2 \in A_c(K)$ such that $f_1\leqslant f_2$. Then for any $\psi \in \partial_{e} K'$,
$$\Phi(f_1)(\psi)=\frac{1}{f_1(\alpha^{-1} (\psi))}\geqslant \frac{1}{f_2(\alpha^{-1} (\psi))}= \Phi(f_2)(\psi).$$ Thus $\Phi(f_1)|_{\partial_{e} K'}\geqslant \Phi(f_2)|_{\partial_{e} K'}$. Hence $\Phi(f_1)\geqslant \Phi(f_2)$.
 Moreover, for $\lambda >0$, $$\Phi(\lambda f_1)(\psi)=\frac{1}{\lambda f_1(\alpha^{-1} (\psi))}=\frac{1}{\lambda} \Phi(f_1)(\psi).$$ Thus $\Phi(\lambda f_1)|_{\partial_{e} K'}= \frac{1}{\lambda} \Phi(f_1)|_{\partial_{e} K'}$. Hence $\Phi(\lambda f_1)= \frac{1}{\lambda} \Phi(f_1)$.
  The map $\Phi$ is bijective with $\Phi^{-1}(g)= \frac{1}{g} \circ \alpha$,  
  and therefore $\phi$ is a gauge reversing bijection. 
\end{proof}


\begin{proof}[Proof of Theorem \ref{precoro}]
The proof follows from arguments entirely analogous to the the proof of Theorem $\ref{coro}$. Specifically, for (i) $\Rightarrow$ (ii), one replaces the application of Propositions $\ref{phiuscmodi}$ and $\ref{philsc}$ with Proposition $\ref{prephilsc}$ to establish the boundary mappings $f \circ \alpha^{-1}$ and $g \circ \alpha$. The (ii) $\Rightarrow$ (i) follows symmetrically by defining $\Phi(f) = f \circ \alpha^{-1}$ and utilizing the order-preserving nature of the extensions to verify that $\Phi$ is a gauge-preserving bijection. 
\end{proof}

\section{Some Applications}\label{Sec6}


\begin{proposition}\label{app1}
Let $K$ be a compact convex set. Suppose $\Phi: A_c(K) \to A_c(K)$ is a gauge-reversing (resp. gauge-preserving) bijecton satisfying $\Phi(\mathbb{1}_K)=\mathbb{1}_{K'}$. Then $\Phi$ is an involution if and only if the correspoding bijection $\alpha: \partial_eK \to \partial_eK$ (see Theorem $\ref{coro}$ and $\ref{precoro}$) is an involution. 
\end{proposition}
\begin{proof}
We present the proof for the gauge-reversing case; the gauge-preserving case follows analogously. By Theorem \ref{coro}, we have
\begin{align*}
& \Phi(f) = \Phi^{-1}(f) && \forall f \in A_c(K) \\
\iff & (\frac{1}{f}\circ \alpha^{-1})(k)= (\frac{1}{f}\circ \alpha)(k) && \forall k \in \partial_e K, \forall f \in A_c(K) \\
\iff & \frac{1}{f(\alpha^{-1}(k))}= \frac{1}{f(\alpha(k))} && \forall k \in \partial_e K, \forall f \in A_c(K) \\
\iff & f(\alpha^{-1}(k))= f(\alpha(k)) && \forall k \in \partial_e K, \forall f \in A_c(K) 
\end{align*}
Since $A(K)=\overline{\text{span}}A_c(K)$, therefore
\begin{align*}
\iff & \alpha^{-1}(k) = \alpha(k) && \forall k \in \partial_e K \\
\iff & \alpha^{-1} = \alpha
\end{align*}
\end{proof}
If $\Phi: A_c(K) \to A_c(K')$ is a gauge-reversing bijection. Then for each $f \in A_c(K)$, the negative Gateaux derivative of $\Phi$ at $f$,  $-D(\Phi(f)): A_c(K) \to A_c(K')$, is a gauge-preserving bijection. Moreover, $-D(\Phi(f))(f)=\Phi(f)$ see \cite[Theorem 4.7]{RT}.

\begin{corollary}
Let $\Phi: A_c(K) \to A_c(K')$ be a gauge-reversing bijection with $\Phi(\mathbb{1}_K)=\mathbb{1}_{K'}$. Let $\alpha, \beta: K \to K'$ be the associated bijections corresponding to $\Phi$ and $-D(\Phi(\mathbb{1}_K))$, respectively, satisfying ($\ref{gr}$) and ($\ref{gp}$). Then $\alpha^{-1}\circ \beta^{-1}= \beta \circ \alpha$.
\end{corollary}
\begin{proof}
Define $\Psi:=-D(\Phi(\mathbb{1}_K))^{-1}\circ \Phi$. By \cite[Lemma 5.9]{RT}, $\Psi$ is
a gauge-reversing involution on $A_c(K)$.  For any $g\in A_c(K)$ and $k\in \partial_eK$, we have
$$\Psi(g)(k)= -D(\Phi(\mathbb{1}_K))^{-1} (\phi(g))(k)=\Phi(g)(\beta^{-1}(k))= \frac{1}{g}\circ (\alpha^{-1}\circ \beta^{-1}) (k).$$ Hence, by Proposition $\ref{app1}$, the desired conclusion follows.
\end{proof}

\begin{Acknowledgement}
The research of the second author is financially supported by the Institute Postdoctoral Fellowship, NISER Bhubaneswar.
\end{Acknowledgement}

\bibliographystyle{plain, abbrv}

\begin{thebibliography}{9999a}

\bibitem{A} Alfsen, E.M.:  Compact convex sets and boundary integrals.  Springer-Verlag, Berlin {\bf 57}, (1971).

\bibitem{ASS} Alfsen, E.M., Shultz, F.W.,  Størmer, E.:  A Gelfand-Neumark theorem for Jordan algebras. Ado. Math. {\bf 28}, 11--56, (1978).

\bibitem{GN} Gelfand, I.M., Naimark, M.A.: On the embedding of normed rings into the ring of operators in Hilbert space. Mat. Sb. {\bf 12}, 87--105, (1943).

\bibitem{K1} Kadison, R.V.:  A representation theory for commutative topological  algebra. Mem. Amer. Math. Soc. {\bf 7}, (1951).

\bibitem{K} Koecher, M.: Positivätsbereiche im $\mathbb{R}^n$. Am. J. Math. {\bf 79}, 575--596, (1957).

\bibitem{LRI} Lemmens, B., Roelands, M., Imhoff, H.B.V.: An order theoretic characterization of spin factors. Q. J. Math. {\bf 68}, 1001--1017, no. 3, (2017).

\bibitem{LRW} Lemmens, B., Roelands, M., Wortel, M.: Infinite dimensional symmetric cones and gauge-reversing maps. arXiv preprint: 2504.12487 (2025).

\bibitem{NS} Noll, W., Schäffer, J.J.:  Orders, gauge, and distance in faceless linear cones; with examples relevant to continuum mechanics and relativity. Arch. Ration. Mech. Anal. {\bf 66}, 343--377, (1977).

\bibitem{RT}  Roelands, M., Tiersma, S.:  An order-theoretic characterization of JB-algebras. arXiv preprint: 2507.09526 (2025).
 

\bibitem{S51} Sherman, S.: Order in operator algebras, Am. J. Math. {\bf 73}, 227--232, no. 1, (1951).

\bibitem{W} Walsh, C.:  Gauge-reversing maps on cones, and Hilbert and Thompson isometries. Geom. Topol. {\bf 22}, 55--104, no. 1, (2018).

\bibitem{V} Vinberg, E. B.:  Homogeneous cones. Sov. Math. Dokl. {\bf 1},  787--790, (1960).
\end{thebibliography}

\end{document}